\documentclass[graybox]{svmult}

\usepackage{mathptmx}       % selects Times Roman as basic font
\usepackage{helvet}         % selects Helvetica as sans-serif font
\usepackage{courier}        % selects Courier as typewriter font
\usepackage{type1cm}        % activate if the above 3 fonts are
\usepackage{makeidx}         % allows index generation
\usepackage{graphicx}        % standard LaTeX graphics tool
\usepackage{multicol}        % used for the two-column index
\usepackage[bottom]{footmisc}% places footnotes at page bottom
\usepackage{bm}
\makeindex             % used for the subject index
\usepackage{latexsym}
\usepackage{dsfont}
\usepackage{bbm}
\usepackage{amssymb}
\usepackage{amsmath}

\newtheorem{Theorem}{Theorem}[section]
\newtheorem{Lemma}[Theorem]{Lemma}

\newtheorem{Prop}[Theorem]{Proposition}
\def\cB{\mathcal{B}}
\def\cC{\mathcal{C}}

\def\cL{\mathcal{L}}

\def\cS{\mathcal{S}}

\def\Erw{\mathbb{E}}

\def\N{\mathbb{N}}
  
\def\Prob{\mathbb{P}} 

\def\R{\mathbb{R}}

\def\U{\mathbb{U}}
\def\V{\mathbb{V}}

\def\Z{\mathbb{Z}}

\def\sfu{\mathsf{u}}
\def\sfv{\mathsf{v}}

\def\eps{\varepsilon}

\def\1{\vec{1}}
\def\3{{\ss}}
\def\llam{\lambda\hspace{-5.1pt}\lambda}

\def\wh{\widehat}

\def\ovl{\overline}

\def\sg{\sigma^{>}}

\def\ij{\,i\!j}
\def\jj{j\!j}

\def\sfu{\mathsf{u}}
\def\sfv{\mathsf{v}}

\allowdisplaybreaks[4] %Formeln dŸrfen umgebrochen werden

\begin{document}

\title*{Quasi-stochastic matrices and Markov renewal theory}
\titlerunning{Quasi-stochastic matrices and Markov renewal theory}
\author{Gerold Alsmeyer}
\institute{Gerold Alsmeyer \at Inst.~Math.~Statistics, Department
of Mathematics and Computer Science, University of M\"unster,
Einsteinstrasse 62, D-48149 M\"unster, Germany.\at
\email{gerolda@math.uni-muenster.de}\\
Research supported by the Deutsche Forschungsgemeinschaft (SFB 878)}

\maketitle

\abstract{Let $\cS$ be a finite or countable set. Given a matrix $F=(F_{\ij})_{i,j\in\cS}$ of distribution functions on $\R$ and a quasi-stochastic matrix $Q=(q_{\ij})_{i,j\in\cS}$, i.e.\ an irreducible nonnegative matrix with maximal eigenvalue 1 and associated unique (modulo scaling) positive left and right eigenvectors $\sfu,\sfv$, the matrix renewal measure $\sum_{n\ge 0}Q^{n}\otimes F^{*n}$ associated with $Q\otimes F:=(q_{\ij}F_{\ij})_{i,j\in\cS}$ (see below for precise definitions) and a related Markov renewal equation are studied. This was done earlier by de Saporta \cite{Saporta:03} and Sgibnev \cite{Sgibnev:01,Sgibnev:06} by drawing on potential theory, matrix-analytic methods and Wiener-Hopf techniques. The purpose of this article is to describe a quite different probabilistic approach which embarks on the observation that $Q\otimes F$ turns into an ordinary semi-Markov matrix after a harmonic transform. This allows us to relate $Q\otimes F$ to a Markov random walk $(M_{n},S_{n})_{n\ge 0}$ with discrete recurrent driving chain $(M_{n})_{n\ge 0}$. It is then shown that renewal theorems including a Choquet-Deny-type lemma may be easily established by resorting to standard renewal theory for ordinary random walks. Three typical examples are presented at the end of the article.
%Motivated by problems in the analysis of random difference equations (perpetuities) in Markovian environment, renewal theorems including a Choquet-Deny-type lemma for transient Markov random walks on the line with finite driving chain have been derived by de Saporta \cite{Saporta:03} by drawing on potential theory and matrix-analytic methods. Similar results have been obtained by Sgibnev \cite{Sgibnev:01,Sgibnev:06} using Wiener-Hopf techniques. A common feature of these works is the appearance of certain nonnegative matrices with spectral radius one which turn into a stochastic matrix after a harmonic transform and are therefore called quasi-stochastic here. Taking the latter as a starting point, the purpose of the present paper is to show how their results may be established by much simpler probabilistic arguments and the use of standard renewal theory for ordinary random walks. This even allows the state space of the driving chain to be countably infinite.
}

\bigskip



\section{Introduction and main results}\label{sec:intro}

Quasi-stochastic matrices (see below for the formal definition) are a generalization of stochastic matrices and thus of transition matrices of Markov chains with countable state space. In applications, three of which may be found in the final section of this article, such matrices appear when studying the limit behavior of certain functionals of processes which are driven by discrete Markov chains. These processes, called \emph{Markov random walks} or \emph{Markov-additive processes}, are characterized by having increments which are conditionally independent given the driving chain. Moreover, the conditional distribution of the $n^{th}$ increment depends only on the state of the chain at time $n-1$ and $n$. Aiming at limit results as just mentioned, our main purpose is to show that via a harmonic transform quasi-stochasticity may easily be reduced to stochasticity and thus to ordinary transition matrices. This in turn allows the use of more intuitive probabilistic arguments instead of analytic ones. Further information will follow below after a description of the basic setup.

\vspace{.2cm}
We proceed with a definition of a quasi-stochasticity.
Let $\cS=\{1,...,m\}$ for some $m\in\N$ or $\cS=\N$. Suppose we are given an irreducible nonnegative matrix $Q=(q_{\ij})_{i,j\in\cS}$ with maximal eigenvalue 1 for which there exist unique positive left and right eigenvectors $\sfu=(\sfu_{i})_{i\in\cS},\sfv=(\sfv_{i})_{i\in\cS}$ modulo scaling, thus
\begin{equation}\label{eigenvalue one}
\sfu^{\top}Q=\sfu^{\top}\quad\text{and}\quad Q\sfv=\sfv.
\end{equation}
A matrix of this kind will be called \emph{quasi-stochastic} hereafter. If $\cS$ is finite or, more generally, $\sum_{i\in\cS}\sfu_{i}<\infty$ and $\sfu^{\top}\sfv<\infty$, strict uniqueness is rendered upon choosing the normalization
\begin{equation}\label{Perron eigenvector normalized}
\sum_{i\in\cS}\sfu_{i}=1\quad\text{and}\quad\sfu^{\top}\sfv=\sum_{i\in\cS}\sfu_{i}\sfv_{i}=1.
\end{equation}
Note that under these assumptions all powers $Q^{n}=(q_{\ij}^{(n)})_{i,j\in\cS}$ are also nonnegative matrices with finite entries (plainly a nontrivial statement only if $\cS$ is infinite).

The example that comes to mind first is when $Q$ equals the transition matrix of a recurrent discrete Markov chain on $\cS$ and thus a proper stochastic matrix for which the left eigenvector $\sfu$ is the essentially unique stationary measure of the chain. In the positive recurrent case, one may choose $\sfu$ as the unique stationary distribution and $\sfv=(1,1,...)^{\top}$.

\vspace{.2cm}
Next, let $F_{\ij}$ for $i,j\in\cS$ be proper distribution functions on $\R$, thus nondecreasing, right continuous with limit 0 at $-\infty$ and 1 at $+\infty$. Define the matrix function
$$ \R\ \ni\ t\ \mapsto\ Q\otimes F(t)\ =\ ((Q\otimes F)_{\ij}(t))_{i,j\in\cS}\ :=\ (q_{\ij}F_{\ij}(t))_{i,j\in\cS}, $$
where $F(t):=(F_{\ij}(t))_{i,j\in\cS}$. If $B(t)=(B_{\ij}(t))_{i,j\in\cS}$ denotes another matrix of real-valued functions, the convolution $(Q\otimes F)*B$ of $Q\otimes F(t)$ and $B(t)$ is defined as
\begin{equation*}
((Q\otimes F)*B)_{\ij}(t)\ :=\ \sum_{k\in\cS}\int_{\R}B_{kj}(t-x)\ (Q\otimes F)_{ik}(dx)\quad (i,j\in\cS,\ t\in\R),
\end{equation*}
provided that the integrals exist. Since
\begin{equation*}
((Q\otimes F)*(Q\otimes F))_{\ij}(t)\ =\ \sum_{k\in\cS}q_{ik}q_{kj}F_{ik}*F_{kj}(t)\ \le\ \sum_{k\in\cS}q_{ik}q_{kj}\ =\ q_{\ij}^{(2)}
\end{equation*}
for all $i,j\in\cS$, we find that $(Q\otimes F)^{*2}$ exists (as a componentwise finite-valued function) and then upon induction over $n$ the very same for $(Q\otimes F)^{*n}$, recursively defined by
$$ (Q\otimes F)^{*n}(t)\ =\ (Q\otimes F)*(Q\otimes F)^{*(n-1)}(t)\quad (t\in\R) $$
for $n\ge 1$, where $A^{*0}(t)$ equals the identity matrix for each $t\ge 0$ and any matrix function $A$. The induction also shows that
\begin{equation*}
(Q\otimes F)^{*n}(t)\ =\ \big(q_{\ij}^{(n)}F_{\ij}^{*n}(t)\big)_{i,j\in\cS}\ =\ Q^{n}\otimes F^{*n}(t)\quad (t\in\R,\,n\in\N_{0}).
\end{equation*}

Of particular interest in this work is the \emph{matrix renewal measure associated with $Q\otimes F$}, viz.
$$ \V((t,t+h]) :=\ \sum_{n\ge 0}\big((Q\otimes F)^{*n}(t+h)-(Q\otimes F)^{*n}(t)\big)\quad (t\in\R,\,h>0) $$
under conditions ensuring that the entries of $\V=(\V_{\ij})_{i,j\in\cS}$ are Radon measures. The matrix measure $\V$ arises in connection with the solution $Z(t)=(Z_{i}(t))_{i\in\cS}$ of a system of renewal equations, namely
\begin{equation*}
Z_{i}(t)\ =\ z_{i}(t)\ +\ \sum_{j\in\cS}q_{\ij}\int_{\R}Z_{j}(t-x)\ F_{\ij}(dx)\quad (t\in\R,\,i\in\cS),
\end{equation*}
shortly written as $Z=z+(Q\otimes F)*Z$, where $z(t)=(z_{i}(t))_{i\in\cS}$ is a vector of real-valued functions. Indeed, if 
$$ Z(t)\ =\ \V*z(t)\ =\ \big(\V_{i}*z(t)\big)_{i\in\cS}\ =\ \left(\sum_{j\in\cS}q_{\ij}^{(n)}\int_{\R}z_{j}(t-x)\ F_{\ij}^{*n}(dx)\right)_{i\in\cS} $$ 
exists for all $t\in\R$, then it forms a solution which is even unique under additional assumptions as we will see later.

Apart from allowing $\cS$ to be infinite, our setup is the same as in the papers by de Saporta \cite{Saporta:03} and Sgibnev \cite{Sgibnev:06} who derive a Blackwell-type renewal theorem for $\V$ and determine the asymptotic behavior of $Z(t)=\V*z(t)$ under appropriate conditions. De Saporta's approach is based on potential theory and rather technical, while Sgibnev uses a matrix-analytic approach in combination with a matrix Wiener-Hopf factorization as described in \cite{Asmussen:89}. The main purpose of this article is to provide a different, purely probabilistic approach within the framework of discrete Markov renewal theory which not only allows us to interpret assumptions in a more natural context but is also considerably simpler. The latter is due to the fact that main results in discrete Markov renewal theory, which deals with random walks driven (or modulated) by a recurrent Markov chain with discrete state space, can be easily deduced from classical renewal theory dealing with ordinary random walks with positive drift. This is done by drawing on stopping times, occupation measures and regeneration techniques and will be demonstrated in Section \ref{sec:DMRT}, for it has apparently never been carried out in the literature (though a similar approach may already be found in the classical paper by Athreya, McDonald and Ney \cite{Athreya+et.al.:78a}). For basic definitions and properties of Markov random walks and Markov renewal processes with discrete driving chain we refer to the textbooks by Asmussen \cite[p. 206ff]{Asmussen:03} and \c{C}inlar \cite[Ch. 10]{Cinlar:75b}, or \cite{Cinlar:69}.

\vspace{.1cm}
Besides quasi-stochasticity, the following two standing assumptions about $Q$ will be made throughout this work:
\begin{equation}\label{eq:recurrence of Q}\tag{A1}
\sum_{n\ge 1}q_{ii}^{(n)}\ =\ \infty\quad\text{for some }i\in\cS.
\end{equation}
\begin{equation}\label{eq:positive drift}\tag{A2}
\mu\ :=\ \sum_{i\in\cS}\sum_{j\in\cS}\sfu_{i}\,q_{\ij}\,\sfv_{j}\int x\ F_{\ij}(dx)\ >\ 0.
\end{equation}
In terms of the stochastic matrix $P$ associated with $Q$, to be introduced in Section \ref{sec:setup} below, condition \eqref{eq:recurrence of Q} means that $P$ is recurrent, while \eqref{eq:positive drift} ensures that the Markov random walk associated with $P\otimes F$ has positive stationary drift (see Lemma \ref{lem:properties MRW}). Since $Q$ (and thus $P$) is irreducible, it follows by solidarity that \eqref{eq:recurrence of Q} actually implies $\sum_{i\in\cS}q_{ii}^{(n)}=\infty$ \emph{for all} $i\in\cS$. Moreover, \eqref{eq:recurrence of Q} automatically holds if $\cS$ is finite.

\vspace{.1cm}
We further need the following lattice-type condition on $Q\otimes F$ which is due to Shurenkov \cite{Shurenkov:84}: $Q\otimes F$ is called $d$-arithmetic, if $d$ is the maximal positive number such that
\begin{equation}\label{eq:lattice condition}
F_{\ij}(\gamma(j)-\gamma(i)+d\Z)=F_{\ij}(\infty)
\end{equation}
for all $i,j\in\cS$ with $\sfu_{i}\,q_{\ij}\,\sfv_{j}>0$ and some measurable $\gamma:\cS\to [0,d)$, called shift function. If no such $d$ exists, $Q\otimes F$ is called nonarithmetic. Notice that \eqref{eq:lattice condition} for all $i,j$ as stated implies
\begin{equation*}
F_{\ij}^{*n}(\gamma(j)-\gamma(i)+d\Z)=F_{\ij}^{*n}(\infty)
\end{equation*}
for all $i,j\in\cS$ with $\sfu_{i}\,q_{\ij}^{(n)}\,\sfv_{j}>0$ and all $n\in\N$. Consequently, if $F_{\ij}^{*n}$ is nonsingular with respect to Lebesgue measure $\llam$ for some $n\in\N$ and $i,j\in\cS$ with $\sfu_{i}\,q_{\ij}^{(n)}\,\sfv_{j}>0$, then $Q\otimes F$ must be nonarithmetic and is called \emph{spread out}. As in the classical renewal setup, this property entails a Stone-type decomposition of the matrix renewal measure $\V$ which in turn leads to some improvements of the renewal results on $\V$ in the nonarithmetic case.

We proceed to the statement of our main results all proofs of which are presented in Section \ref{sec:proofs}. For the sake of brevity we restrict ourselves to the case of nonarithmetic $Q\otimes F$ but note that all given results have obvious arithmetic counterparts which are obtained in a similar manner.

% Hier main results

If $\cS$ is finite, the following result is Theorem 3 in \cite{Saporta:03} and Theorem 1 in \cite{Sgibnev:06}. 

\begin{Theorem}\label{MRT for Q}
Let $Q$ be a quasi-stochastic matrix satisfying \eqref{eq:recurrence of Q} and \eqref{eq:positive drift} and suppose that $Q\otimes F$ is nonarithmetic. Then the associated renewal measure $\V$ satisfies
\begin{equation*}
\lim_{t\to\infty}\V_{\ij}((t,t+h])\ =\ \frac{\sfv_{i}\,\sfu_{j}h}{\mu}\quad\text{and}\quad\lim_{t\to-\infty}\V_{\ij}((t,t+h])\ =\ 0
\end{equation*}
for all $h>0$ and $i,j\in\cS$.
\end{Theorem}

The next result provides a Stone-type decomposition of $\V$ that for finite $\cS$ was derived by other means in \cite[Theorem 2]{Sgibnev:01} (one-sided case) and \cite[Theorem 5]{Sgibnev:06}.

\begin{Theorem}\label{Stone-type & MRT for Q}
Let $Q$ be a quasi-stochastic matrix satisfying \eqref{eq:recurrence of Q} and \eqref{eq:positive drift} and suppose that $Q\otimes F$ is spread out. Then the associated renewal measure allows a Stone-type decomposition $\V=\V^{1}+\V^{2}$, where
\begin{description}[xxx]
\item[(a)] $\V^{1}=(\V_{\ij}^{1})_{i,j\in\cS}$ consists of finite measures $\V_{\ij}^{1}\,$,\vspace{.05cm}
\item[(b)] $\V^{2}=(\V_{\ij}^{2})_{i,j\in\cS}$ consists of $\llam$-continuous measures $\V_{\ij}^{2}$ with densities $h_{\ij}$ that are bounded, continuous and satisfy
$$ \lim_{t\to\infty}h_{\ij}(t)\ =\ \frac{\sfv_{i}\sfu_{j}}{\mu}\quad\text{and}\quad\lim_{t\to\infty}h_{\ij}(t)\ =\  0 $$
for all $i,j\in\cS$.
\end{description}
Furthermore,
\begin{equation*}
\lim_{t\to\infty}\sup_{\cB(\R)\ni B\subset [0,h]}\left|\V_{\ij}(B)-\frac{\sfv_{i}\sfu_{j}\llam(B)}{\mu}\right|\ =\ 0
\end{equation*}
for all $h>0$ and $i,j\in\cS$.
\end{Theorem}

Turning to the functional version of the two previous results, consider a positive sequence $\lambda=(\lambda_{i})_{i\in\cS}$ and a measurable function $g:\cS\times\R\to\R$. The function $g$ is called \emph{$\lambda$-directly Riemann integrable} if
\begin{align}
&g_{i}\text{ is $\llam$-almost everwhere continuous for all }i\in\cS,\label{eq:dRi1}\\
&\hspace{.02cm}
\sum_{i\in\cS}\lambda_{i}\sum_{n\in\Z}\ \sup_{n\eps<x\le (n+1)\eps}|g_{i}(x)|\ <\ \infty\text{ for some }\eps>0,\label{eq:dRi2}
\end{align}
where $g_{i}:=g(i,\cdot)$. If $\cS$ is finite, then this reduces to the statement that $g_{i}$ for each $i\in\cS$ is directly Riemann integrable in the ordinary sense and the following result reduces to Theorem 4 in both, \cite{Saporta:03} and \cite{Sgibnev:06}, for the general nonarithmetic case. For the spread-out case see also \cite[Theorem 3]{Sgibnev:01} and \cite[Theorem 6]{Sgibnev:06}.

\begin{Theorem}\label{MRT2 for Q}
Under the same assumptions as in Theorem \ref{MRT for Q}, let $g$ be $\sfu$-directly Riemann integrable. Then $\V*g=(\V_{i}*g(t))_{i\in\cS}$ has bounded components, i.e.
\begin{equation*}
\sup_{t\in\R}|\V_{i}*g(t)|\ <\ \infty\quad\text{for all }i\in\cS,
\end{equation*}
and, furthermore,
\begin{equation*}
\lim_{t\to\infty}(\V*g)_{i}(t)\ =\ \frac{\sfv_{i}}{\mu}\sum_{j\in\cS}\sfu_{j}\int g_{j}(x)\ dx\quad\text{and}\quad\lim_{t\to-\infty}(\V*g)_{i}(t)\ =\ 0
\end{equation*}
for all $i\in\cS$. If $Q\otimes F$ is even spread out, then the assertions remain valid for all functions $g$ satisfying
\begin{align}
&g_{i}\in L^{\infty}(\llam)\text{ and }\lim_{|x|\to\infty}g_{i}(x)=0\ \text{ for all }i\in\cS,\label{eq:dRi3} \\
&\hspace{2cm}\sum_{i\in\cS}\sfu_{i}\,\|g_{i}\|_{\infty}\ <\ \infty,\label{eq:dRi4}\\
&\hspace{.6cm}g\in L^{1}(\sfu\otimes\llam),\text{ i.e. }\sum_{i\in\cS}\sfu_{i}\|g_{i}\|_{1}\ <\ \infty.\label{eq:dRi5}
\end{align}
\end{Theorem}

Turning finally to the Markov renewal equation $Z=z+(Q\otimes F)*Z$, it is now relatively easy to provide conditions such that $Z^{*}=\V*z$ forms a solution. On the other hand, the question of uniqueness of $Z^{*}$ within a reasonable class of functions appears to be  more difficult, especially if the state space $\cS$ of the driving chain is infinite. Conditions that guarantee uniqueness are often hard to verify in concrete applications.

\vspace{.1cm}
Given any $Z:\cS\times\R\to\R$, let $\wh{Z}:=D^{-1}Z=(\sfv_{i}^{-1}Z_{i})_{i\in\cS}$. Then define
\begin{align*}
\cL\ :=\ &\{Z:\ \|\wh Z_{i}\|_{\infty}<\infty\text{ and }\lim_{t\to-\infty}\wh Z_{i}(t)=0\text{ for all }i\in\cS\},\\
%\cL_{0}\ :=\ &\{Z\in\cL:\ \lim_{t\to-\infty}\wh Z_{i}(t)=0\text{ for all }i\in\cS\},\\
\cL_{0}\ :=\ &\{Z\in\cL:\ \sup_{i\in\cS}\|\wh Z_{i}\|_{\infty}<\infty\},\\
\cL_{0}(g)\ :=\ &\{Z:\ \wh{Z}-\wh{g}\in\cL_{0}\},\\
\cC_{b}\ :=\ &\{Z:\ \sup_{i\in\cS}\|\wh Z_{i}\|_{\infty}<\infty\text{ and }Z_{i}\text{ is continuous for all }i\in\cS\}.
\end{align*}
Note that $\cL=\cL_{0}$ if $\cS$ is finite.

\begin{Theorem}\label{MRE}
Let $Q$ be a quasi-stochastic matrix satisfying \eqref{eq:recurrence of Q} and \eqref{eq:positive drift} and suppose that $Q\otimes F$ is nonarithmetic. Let further $z:\cS\times\R\to\R$ be $\sfu$-directly Riemann integrable, or satisfy conditions \eqref{eq:dRi3}--\eqref{eq:dRi5} if $Q\otimes F$ is even spread out. Then $Z^{*}=\V*z$ is an element of $\cL$ and the unique solution to $Z=z+(Q\otimes F)*Z$ in $\cL_{0}(Z^{*})$. It is also the unique solution in the larger class $\cL_{0}$ if $\cS$ is finite or, more generally, $Z^{*}\in\cL_{0}$.
\end{Theorem}

Note that within the class of component-wise bounded functions there are in fact infinitely many solutions to $Z=z+(Q\otimes F)*Z$, namely all functions 
$$ Z^{c}(t)\ :=\ \V*z(t)+c\sfv\ =\ \big(\V_{i}*z(t)+c\sfv_{i}\big)_{i\in\cS}\quad (t\in\R) $$ 
for $c\in\R$. This means that the constant vectors $c\sfv=(c\sfv_{i})_{i\in\cS}$ are solutions to the homogeneous (Choquet-Deny type) equation $Z=(Q\otimes F)*Z$. The following theorem further shows that they are in fact the only ones within the class $\cC_{b}$. If $\cS$ is finite, this was established analytically by de Saporta \cite[Subsection 3.2]{Saporta:03} extending earlier results by Crump \cite{Crump:70} and Athreya and Rama Murthy \cite{AthreyaRama:76} in the one-sided case when all $z_{i},Z_{i}$ and/or $F_{\ij}$ are concentrated on $[0,\infty)$. Not necessarily continuous solutions in the one-sided case are also discussed in some detail by Cinlar \cite[Sections 3 and 4]{Cinlar:69} in his survey of Markov renewal theory. For yet another and quite recent extension of these results see \cite{Sgibnev:10}. Here we give a simple probabilistic argument which essentially reduces the problem to the classical renewal setup where the answer is known (see \cite[p.\ 382]{Feller:71}).

\begin{Theorem}\label{Choquet-Deny}
Let $Q$ be a quasi-stochastic matrix satisfying \eqref{eq:recurrence of Q} and \eqref{eq:positive drift} and suppose that $Q\otimes F$ is nonarithmetic. Then any solution $Z\in\cC_{b}$ to the equation $Z=(Q\otimes F)*Z$ equals $c\sfv$ for some $c\in\R$.
\end{Theorem}

\section{The Markov renewal setup}\label{sec:setup}

Put $D:=\text{diag}(\sfv_{i},i\in\cS)$ and $\pi=(\pi_{i})_{i\in\cS}$ with $\pi_{i}:=\sfu_{i}\sfv_{i}$ for $i\in\cS$. By \eqref{Perron eigenvector normalized}, $\pi$ defines a probability distribution on $\cS$ if both, the $\sfu_{i}$ and $\sfu_{i}\sfv_{i}$ are summable. Put further
$$ P\ :=\ D^{-1}QD\ =\ \left(\frac{q_{\ij}\sfv_{j}}{\sfv_{i}}\right)_{i,j\in\cS} $$
which forms an irreducible \emph{stochastic matrix} having essentially unique left eigenvector $\pi=\sfu^{\top}D=(\sfu_{i}\sfv_{i})_{i\in\cS}$ associated with its maximal eigenvalue 1. Then
\begin{equation}\label{eq:PotimesF(t)}
\Lambda(t)\ :=\ P\otimes F(t)\ =\ D^{-1}(Q\otimes F)(t)D\ =\ \left(\frac{q_{\ij}F_{\ij}(t)\sfv_{j}}{\sfv_{i}}\right)_{i,j\in\cS}
\end{equation}
defines a matrix transition function of a Markov modulated sequence $(M_{n},X_{n})_{n\ge 0}$ with state space $\cS\times\R$. This means that the latter sequence forms a temporally homogeneous Markov chain satisfying
$$ \Prob(M_{n+1}=j,X_{n+1}\le t|M_{n}=i)\ =\ p_{\ij}F_{\ij}(t) $$
for all $n\in\N_{0}$, $i,j\in\cS$ and $t\in\R$. Equivalently, $M=(M_{n})_{n\ge 0}$ forms a Markov chain on $\cS$ with transition matrix $P$ and the $X_{n}$ are conditionally independent given $M$ with
$$ \Prob(X_{n}\le t|M)\ =\ \Prob(X_{n}\le t|M_{n-1},M_{n})\ =\ F_{M_{n-1}M_{n}}(t) $$
for all $n\in\N$ and $t\in\R$. The Markov-additive process associated with $(M_{n},X_{n})_{n\ge 0}$, called Markov random walk (MRW) hereafter, is defined as $(M_{n},S_{n})_{n\ge 0}$, where $S_{n}=X_{0}+...+X_{n}$ for $n\in\N_{0}$. Its occupation measure on $\cS\times\R$ under $\Prob_{i}:=\Prob(\cdot|M_{0}=i)$, called Markov renewal measure, is given by
\begin{equation}\label{eq:def MRM}
\U_{i}(C)\ :=\ \Erw_{i}\left(\sum_{n\ge 0}\1_{C}(M_{n},S_{n})\right)\ =\ \sum_{n\ge 0}\Prob_{i}((M_{n},S_{n})\in C)
\end{equation}
for measurable subsets $C$ of $\cS\times\R$. Since $\cS$ is countable, there is a one-to-one correspondence between the vector measure $(\U_{i})_{i\in\cS}$ and the matrix renewal measure $\U=(\U_{\ij})_{i,j\in\cS}$, where
\begin{equation*}
\U_{\ij}(B)\ :=\ \Erw_{i}\left(\sum_{n\ge 0}\1_{\{M_{n}=j,S_{n}\in B\}}\right)\ =\ \sum_{n\ge 0}\Prob_{i}(M_{n}=j,S_{n}\in B)\quad (B\in\cB(\R)).
\end{equation*}

\begin{Lemma}\label{lem:properties MRW}
Let $Q$ be a quasi-stochastic matrix satisfying \eqref{eq:recurrence of Q} and \eqref{eq:positive drift}. Then the associated MRW $(M_{n},S_{n})_{n\ge 0}$ has recurrent driving chain with stationary measure $\pi$ and positive stationary drift $\mu$ defined in \eqref{eq:positive drift}, thus $\Erw_{\pi}X_{1}=\mu$.
\end{Lemma}

\begin{proof}
Obviously, \eqref{eq:recurrence of Q} is equivalent to
$$ \sum_{n\ge 1}p_{ii}^{(n)}\ =\ \infty\quad\text{for some }i\in\cS $$
which in turn is equivalent to the recurrence of $(M_{n})_{n\ge 0}$ as claimed. The drift assertion follows from
\begin{align*}
\Erw_{\pi}X_{1}\ &=\ \sum_{i,j\in\cS}\Prob_{\pi}(M_{0}=i,M_{1}=j)\,\Erw(X_{1}|M_{0}=i,M_{1}=j)\\
&=\ \sum_{i,j\in\cS}\pi_{i}p_{\ij}\,\int x\ F_{\ij}(dx)
\end{align*}
in combination with the definitions of the $\pi_{i}$ and $p_{\ij}$.\qed
\end{proof}

\begin{Lemma}\label{lem:MRM connection}
Let $Q$ be a quasi-stochastic matrix satisfying \eqref{eq:recurrence of Q} and \eqref{eq:positive drift}. Then 
\begin{equation}\label{eq:V<->U}
\V\ =\ D\,\U\,D^{-1}\ =\ \left(\frac{\sfv_{i}\,\U_{\ij}}{\sfv_{j}}\right)_{i,j\in\cS}
\end{equation}
\end{Lemma}

\begin{proof}
For all $i,j\in\cS$, $t\in\R$ and $h>0$, we have that
\begin{align*}
\U_{\ij}((t,t+h])\ &=\ \sum_{n\ge 0}\Prob_{i}(M_{n}=j,S_{n}\in (t,t+h])\\
&=\ \sum_{n\ge 0}p_{\ij}^{(n)}\big(F_{\ij}^{*n}(t+h)-F_{\ij}^{*n}(t)\big)
\end{align*}
and therefore, using \eqref{eq:PotimesF(t)},
\begin{align*}
\U((t,t+h])\ &=\ \sum_{n\ge 0}\big((P\otimes F)^{*n}(t+h)-(P\otimes F)^{*n}(t)\big)\\
&=\ \sum_{n\ge 0}D^{-1}\big((Q\otimes F)^{*n}(t+h)-(Q\otimes F)^{*n}(t)\big)D\\
&=D^{-1}\left(\sum_{n\ge 0}\big((Q\otimes F)^{*n}(t+h)-(Q\otimes F)^{*n}(t)\big)\right)D\\
&=\ D^{-1}\V((t,t+h])\,D.
\end{align*}
This proves the assertion.\qed
\end{proof}

Eq.\ \eqref{eq:V<->U} provides the crucial relation between the renewal measure $\V$ associated with $Q\otimes F$ and the matrix renewal measure $\U$ whose entries $\U_{\ij}$ are actually ordinary renewal measures as will be shown in Lemma \ref{lem:MRM to ordinary RM}. As a consequence, any result valid for $\U$ is now easily converted into a result for $\V$.

\section{Discrete Markov renewal theory: a purely probabilistic approach}\label{sec:DMRT}

Throughout this section, let $(M_{n},S_{n})_{n\ge 0}$ be an arbitrary nonarithmetic MRW with discrete recurrent driving chain $M=(M_{n})_{n\ge 0}$ having state space $\cS$, transition matrix $P=(p_{\ij})_{i,j\in\cS}$ and stationary measure $\pi=(\pi_{i})_{i\in\cS}$, the latter being unique up to positive scalars. We denote by $X_{1},X_{2},...$ the increments of $(S_{n})_{n\ge 0}$ and by $F_{\ij}$ the conditional distribution of $X_{n}$ given $M_{n-1}=i$ and $M_{n}=j$ for $i,j\in\cS$. Put $\Prob_{i}:=\Prob(\cdot|M_{0}=i)$ with expectation operator $\Erw_{i}$ and let $S_{0}=0$ a.s.\ under $\Prob_{i}$ for each $i\in\cS$. Finally assume that the MRW has positive stationary drift $\mu$, given by
$$ \mu\ =\ \sum_{i\in\cS}\sum_{j\in\cS}\pi_{i}p_{\ij}\mu_{\ij}\ =\ \Erw_{\pi}X_{1}, $$
where $\mu_{\ij}:=\int x\,F_{\ij}(dx)$. Notice that $\mu$, as $\pi$, is only unique up to positive scalars.

\subsection{Auxiliary lemmata}\label{subsec:3.1}

Let $i\in\cS$ be arbitrary but fixed throughout this subsection. Then, we may define $\pi$ as
\begin{equation}\label{eq:def pi}
\pi_{k}\ :=\ \pi_{k}^{(i)}\ :=\ \Erw_{i}\left(\sum_{n=1}^{\sigma(i)}\1_{\{M_{n}=k\}}\right)\quad (k\in\cS),
\end{equation}
where $\sigma(i)$ denotes the first return time of $M$ to $i$. With this choice, we have $\pi_{i}=1$ and may also easily deduce that
\begin{equation}\label{eq:occupation measure formula}
\Erw_{i}\left(\sum_{n=1}^{\sigma(i)}g(M_{n},X_{n})\right)\ =\ \Erw_{i}\left(\sum_{n=0}^{\sigma(i)-1}g(M_{n},X_{n})\right)\ =\ \Erw_{\pi}g(M_{1},X_{1})
\end{equation}
whenever $\Erw_{\pi}g(M_{1},X_{1})$ exists. Note that $\pi^{(j)}=c_{j}\pi^{(i)}$ for any $j\in\cS$ together with $c_{j}\pi_{j}=c_{j}\pi_{j}^{(i)}=\pi_{j}^{(j)}=1$ implies
$c_{j}=\pi_{j}^{-1}$.

If $(\sigma_{n}(i))_{n\ge 1}$ denotes the renewal sequence of successive return times of $M$ to $i$, thus $\sigma(i)=\sigma_{1}(i)$, then $(S_{\sigma_{n}(i)})_{n\ge 1}$ is an ordinary random walk under any $\Prob_{j}$ with increment distribution $\Prob_{i}(S_{\sigma(i)}\in\cdot)$ and drift
$$ \Erw_{i}S_{\sigma(i)}\ =\ \Erw_{i}\left(\sum_{n=1}^{\sigma(i)}X_{n}\right)\ =\ \Erw_{\pi}X_{1}\ =\ \mu, $$
where \eqref{eq:occupation measure formula} has been utilized. In particular, $(S_{\sigma_{n}(i)})_{n\ge 0}$ with $\sigma_{0}(i):=0$ forms a zero-delayed random walk under $\Prob_{i}$. The drift of any other $(S_{\sigma_{n}(j)})_{n\ge 1}$ in terms of $\mu$ and $\pi$ is given in the next lemma.

\begin{Lemma}\label{lem:Erw S_sigma(j)}
For each $j\in\cS$,
$$ \Erw_{j}S_{\sigma(j)}\ =\ \frac{\mu}{\pi_{j}}. $$
\end{Lemma}

\begin{proof}
This follows from
$$ \Erw_{j}S_{\sigma(j)}\ =\ \Erw_{\pi^{(j)}}X_{1}\ =\ \pi_{j}^{-1}\,\Erw_{\pi}X_{1} $$
valid for any $j\in\cS$.\qed
\end{proof}

The following lemma on the lattice-type of the $(S_{\sigma_{n}(j)})_{n\ge 1}$, $j\in\cS$, is stated without proof, which may be accomplished with the help of Fourier transforms.

\begin{Lemma}\label{lem:lattice-type}
Under the stated assumptions, $S_{\sigma(j)}$ is nonarithmetic under $\Prob_{j}$, for any $j\in\cS$.
\end{Lemma}

The next lemma confirms that the Markov renewal measure $\U_{i}$ is directly related to the ordinary renewal measures of the $(S_{\sigma_{n}(j)})_{n\ge 1}$, $j\in\cS$, under $\Prob_{i}$.

\begin{Lemma}\label{lem:MRM to ordinary RM}
For all $j\in\cS$, $\U_{i}(\{j\}\times\cdot)=\U_{\ij}$ equals the (ordinary) renewal measure of $(S_{\sigma_{n}(j)})_{n\ge 1}$ under $\Prob_{i}$ if $j\ne i$, and of $(S_{\sigma_{n}(i)})_{n\ge 0}$ under $\Prob_{i}$ if $j=i$.
\end{Lemma}

\begin{proof}
The assertion follows directly from
\begin{align*}
\U_{\ij}(B)\ &=\ \Erw_{i}\left(\sum_{n\ge 0}\1_{\{M_{n}=j,S_{n}\in B\}}\right)\ =\ 
\begin{cases}
\Erw_{i}\left(\sum_{n\ge 1}\1_{\{S_{\sigma_{n}(j)}\in B\}}\right),&\text{if }j\ne i,\vspace{.1cm}\\
\Erw_{i}\left(\sum_{n\ge 0}\1_{\{S_{\sigma_{n}(i)}\in B\}}\right),&\text{otherwise}.
\end{cases}
\end{align*}
for all $B\in\cB(\R)$.\qed
\end{proof}

For the next result, we define the pre-$\sigma(i)$ occupation measure
\begin{equation*}
\vec{U}_{i}(C)\ :=\ \Erw_{i}\left(\sum_{n=0}^{\sigma(i)-1}\1_{C}(M_{n},S_{n})\right)
\end{equation*}
for measurable subsets $C$ of $\cS\times\R$. Choosing $C=\{j\}\times\R$, we find that
\begin{equation}\label{eq:V_i(j times R)}
\vec{U}_{i}(\{j\}\times\R)\ =\ \Erw_{i}\left(\sum_{n=0}^{\sigma(i)-1}\1_{\{M_{n}=j\}}\right)\ =\ \pi_{j}
\end{equation}
for all $j\in\cS$.

\begin{Lemma}\label{lem:MRM occupation measure formula}
Under the stated assumptions,
\begin{equation}\label{eq:U*V formula1}
\U_{i}(C)\ =\ \sum_{j\in\cS}\ \iint\1_{C}(j,x+y)\ \vec{U}_{i}(\{j\}\times dy)\ \U_{\,ii}(dx)
\end{equation}
for any measurable $C\subset\cS\times\R$, in particular
\begin{equation}\label{eq:U*V formula2}
\U_{\ij}(B)\ =\ \int\vec{U}_{i}(\{j\}\times B-x)\ \U_{\,ii}(dx)\ =\ \int\U_{\,ii}(B-x)\ \vec{U}_{i}(\{j\}\times dx)
\end{equation}
for all $j\in\cS$ and $B\in\cB(\R)$.
\end{Lemma}

\begin{proof}
Writing
\begin{align*}
\U_{i}(C)\ =\ \Erw_{i}\left(\sum_{n\ge 0}\sum_{k=0}^{\sigma_{n+1}(i)-\sigma_{n}(i)-1}\1_{C}(M_{\sigma_{n}(i)+k},S_{\sigma_{n}(i)+k})\right)
\end{align*}
the assertion follows by a standard conditioning argument.\qed
\end{proof}

\begin{Lemma}\label{lem:MRM locally finite}
Under the stated assumptions,
\begin{equation*}
\sup_{t\in\R}\,\U_{\ij}([t,t+h])\ \le\ \pi_{j}\,\U_{\,ii}([-h,h])
\end{equation*}
for all $j\in\cS$ and $h>0$.
\end{Lemma}

\begin{proof}
It is well-known from ordinary renewal theory that
$$ \sup_{t\in\R}\,\U_{\,ii}([t,t+h])\ \le\ \U_{\,ii}([-h,h]) $$
for any $h>0$. Using this and \eqref{eq:V_i(j times R)} in combination with Lemma \ref{lem:MRM occupation measure formula}, we obtain as claimed
\begin{align*}
\U_{\ij}([t,t+h])\ &=\ \int\U_{\,ii}[t-x,t+h-x])\ \vec{U}_{i}(\{j\}\times dx)\\
&\le\ \vec{U}_{i}(\{j\}\times\R)\,\U_{\,ii}[-h,h])\\
&=\ \pi_{j}\,\U_{\,ii}([-h,h])
\end{align*}
for all $j\in\cS$, $t\in\R$ and $h>0$.\qed
\end{proof}

\subsection{Markov renewal theorems}\label{subsec:3.2}

It is now fairly straightforward to derive the Markov renewal theorem in the present setup by drawing on Blackwell's renewal theorem and the key renewal theorem from standard renewal theory. Since $\pi$ is generally unique only up to positive scalars, it should be observed that
$\pi(\cdot)/\mu$ with $\mu$ defined by \eqref{eq:positive drift} does not depend on the particular choice of $\pi$.

\begin{Theorem}[Markov renewal theorem I]\label{thm:MRT1}
Under the assumptions stated at the beginning of this section,
\begin{equation*}
\lim_{t\to\infty}\U_{i}(A\times [t,t+h])\ =\ \frac{\pi(A)\,h}{\mu}\quad\text{and}\quad\lim_{t\to-\infty}\U_{i}(A\times [t,t+h])\ =\ 0
\end{equation*}
for all $i\in\cS$, $\pi$-finite $A\subset\cS$ and $h>0$.
\end{Theorem}

\begin{proof}
This is now a direct consequence of Blackwell's renewal theorem (applied to the $\U_{\ij})$ and the dominated convergence theorem, when using that
$$ \U_{i}(A\times [t,t+h])\ =\ \sum_{j\in A}\U_{\ij}([t,t+h])\ =\ \sum_{j\in A}\U_{\ij}([t,t+h]) $$
by Lemma \ref{lem:MRM to ordinary RM}, that $\sum_{j\in A}\U_{\ij}([t,t+h])\le\pi(A)\,\U_{\,ii}([-h,h])$ (Lemma \ref{lem:MRM locally finite}), and finally
$$ \lim_{t\to\infty}\U_{\ij}([t,t+h])\ =\ \frac{1}{\Erw_{j}S_{\sigma(j)}}\ =\ \frac{\pi_{j}}{\mu} $$
for any $j\in\cS$ (Lemma \ref{lem:Erw S_sigma(j)}).\qed
\end{proof}

Turning to the functional version of the previous result, recall from \eqref{eq:dRi1} and \eqref{eq:dRi2} the definition of a $\pi$-directly Riemann integrable function $g$. The asymptotic behavior of
$$ \U_{i}*g(t)\ =\ \sum_{j\in\cS}\int g_{j}(t-x)\ \U_{\ij}(dx) $$
for any such $g$ and $i,j\in\cS$ is described by the second Markov renewal theorem:

\begin{Theorem}[Markov renewal theorem II]\label{thm:MRT2}
Under the assumptions stated at the beginning of this section, $\U_{i}*g$ is a bounded function satisfying
\begin{align}
&\lim_{t\to\infty}\U_{i}*g(t)\ =\ \frac{1}{\mu}\sum_{j\in\cS}\pi_{j}\int g_{j}(x)\ dx\quad\text{and}\label{eq:MRT2a}\\
&\hspace{1.8cm}\lim_{t\to-\infty}\U_{i}*g(t)\ =\ 0.\label{eq:MRT2b}
\end{align}
for any $\pi$-directly Riemann integrable function $g$ and $i\in\cS$.
\end{Theorem}

\begin{proof}
W.l.o.g.\ let $g$ be nonnegative. Define
$$ \ovl{g}_{s}^{\rho}(t)\ :=\ \sum_{n\in\Z}\bigg(\sup_{n\rho<x\le (n+1)\rho}g_{s}(x)\bigg)\1_{(n\rho,(n+1)\rho]}(t)\quad ((s,t)\in\cS\times\R), $$
for $\rho>0$. Since all $\pi_{i}$ are positive, condition \eqref{eq:dRi2} ensures that $\ovl{g}_{i}^{\,\eps}$ is a directly Riemann integrable (in the ordinary sense) majorant of $g_{i}$
for each $i\in\cS$, which in combination with \eqref{eq:dRi1} implies that $g_{i}$ for any $i$ is directly Riemann integrable as well. Hence, by the key renewal theorem,
$$ \lim_{t\to\infty}\U_{\ij}*g_{j}(t)\ =\ \frac{\pi_{j}}{\mu}\quad\text{and}\quad\lim_{t\to-\infty}\U_{\ij}*g_{j}(t)\ =\ 0 $$
for all $i,j\in\cS$. Now fix any $i\in\cS$ and choose $\pi=\pi^{(i)}$, thus $\pi_{i}=1$. Use Lemma \ref{lem:MRM locally finite} together with \eqref{eq:dRi2} to infer
\begin{align}
\sum_{j\in\cS}\U_{\ij}*g(t)\ &\le\ \sum_{j\in\cS}\U_{\ij}*\ovl{g}^{\,\eps}(t)\nonumber\\
&=\ \sum_{j\in\cS}\sum_{n\in\Z}\bigg(\sup_{n\eps<x\le (n+1)\eps}g_{j}(x)\bigg)\,\U_{\ij}([t-(n+1)\eps,t-n\eps))\nonumber\\
&\le\ \U_{\,ii}([-\eps,\eps])\sum_{j\in\cS}\pi_{j}\sum_{n\in\Z}\bigg(\sup_{n\eps<x\le (n+1)\eps}g_{j}(x)\bigg)\ <\ \infty.\label{eq1:g*U bounded}
\end{align}
Since, by Lemma \ref{lem:MRM to ordinary RM}, we further have that
\begin{align}\label{eq2:g*U bounded}
\U_{i}*g(t)\ &=\ \sum_{j\in\cS}\int g_{j}(t-x)\ \U_{\ij}(dx)\ =\ \sum_{j\in\cS}\U_{\ij}*g(t)
\end{align}
the convergence assertions now follow by an appeal to the dominated convergence theorem.
A combination of \eqref{eq1:g*U bounded} and \eqref{eq2:g*U bounded} further shows the boundedness of $\U_{i}*g$ for each $i\in\cS$.\qed
\end{proof}

\subsection{The spread-out case: a Stone-type decomposition}\label{subsec:3.3}

The final subsection deals with the situation when $(M_{n},S_{n})_{n\ge 0}$ is \emph{spread out} which means that some convolution power of $\Prob_{\pi}(X_{1}\in\cdot)$ is nonsingular with respect to $\llam$ or, equivalently, that $F_{rs}^{*n}$ is nonsingular with respect to $\llam$ for some $r,s\in\cS$ with $p_{rs}^{(n)}>0$ and some $n\in\N$. In this case, Lemma \ref{lem:MRM to ordinary RM} further allows us to derive a Stone-type decomposition of the Markov renewal measure in a very straightforward manner. We begin with a preliminary result on the ordinary renewal measures $\U_{\ij}$.

\begin{Prop}\label{prop:Stone-type U_ij}
Let $(M_{n},S_{n})_{n\ge 0}$ be spread out. Then there exist finite measures $\U_{\ij}^{1}$ and $\llam$-continuous measures $\U_{\ij}^{2}=\upsilon_{\ij}\llam$ such that the following assertions hold for all $i,j\in\cS$:
\begin{description}[(c)]
\item[(a)] $\U_{\ij}=\U_{\ij}^{1}+\U_{\ij}^{2}$.\vspace{.05cm}
\item[(b)] If $i\ne j$, then $\U_{\ij}^{1}=F_{\ij}*\U_{\jj}^{1}$, $\U_{\ij}^{2}=F_{\ij}*\U_{\jj}^{2}$ and $\upsilon_{\ij}=F_{\ij}*\upsilon_{\jj}$.\vspace{.1cm}
\item[(c)] $\upsilon_{\ij}$ is continuous and bounded (uniformly in $i\in\cS$) with
$$ \lim_{t\to\infty}\upsilon_{\ij}(t)\ =\ \frac{\pi_{j}}{\mu}\quad\text{and}\quad\lim_{t\to-\infty}\upsilon_{\ij}(t)\ =\ 0. $$
\end{description}
\end{Prop}

\begin{proof}
Pick any $i\in\cS$. If $r,s\in\cS$ are such that $F_{rs}^{*n}$ has a convolution power is nonsingular with respect to $\llam$ for some $n\in\N$, then choose a cyclic path $(i,r_{1},...,r_{m},i)$ of positive probability $p$ that passes through $r$ and $s$ at consecutive times. This is possible because $(M_{n})_{n\ge 0}$ is irreducible and $p_{rs}>0$. It follows that
\begin{align*}
G_{ii}\ :=\ \Prob_{i}^{S_{\sigma_{1}(i)}}\ &=\ p_{ii}F_{ii}+\sum_{n\ge 2}\ \sum_{i_{1},...,i_{n-1}\in\cS\backslash\{i\}}p_{ii_{1}}\cdot...\cdot p_{i_{n-1}i}\,F_{ii_{1}}*...*F_{i_{n-1}i}\\
&\ge\ p\,F_{ir_{1}}*...*F_{rs}*...*F_{r_{n}i}
\end{align*}
and hence that $G_{ii}$ is spread out. Consequently, Stone's decomposition for ordinary renewal measures provides us with $\U_{\,ii}=\U_{\,ii}^{1}+\U_{\,ii}^{2}$ for some finite measure $\U_{\,ii}^{1}$ and some $\llam$-continuous measure $\U_{\,ii}^{2}=\upsilon_{ii}\llam$ such that $\upsilon_{ii}$ is bounded and continuous with limit 0 at $-\infty$ and
$$ \lim_{t\to\infty}\upsilon_{ii}(t)\ =\ \frac{1}{\Erw_{i}S_{\sigma(i)}}\ =\ \frac{\pi_{i}}{\mu}. $$
All remaining assertions are now easily derived when using $\U_{\ij}=F^{\ij}*\U_{\jj}$ for $i,j\in\cS$ with $i\ne j$. Further details are therefore omitted.\qed
\end{proof}

It is now very easy to further obtain a Stone-type decomposition of the Markov renewal measures $\U_{i}$ for $i\in\cS$.

\begin{Theorem}[Stone-type decomposition]\label{thm:MRM spread out}
Let $(M_{n},S_{n})_{n\ge 0}$ be spread out. Then the following assertions hold true for each $i\in\cS$: There exists a finite measure $\U_{i}^{1}$ and a $\pi\otimes\llam$-continuous measure $\U_{i}^{2}$ with density $\upsilon_{i}$ such that
\begin{description}[(b)]
\item[(a)] $\U_{i}=\U_{i}^{1}+\U_{i}^{2}$.
\item[(b)] $\upsilon_{i}$ is bounded on any $\cS_{0}\times\R$ with $\sup_{i\in\cS_{0}}\pi_{i}<\infty$.
\item[(c)] $\upsilon_{\ij}(\cdot):=\upsilon_{i}(j,\cdot)$ is continuous for any $j\in\cS$ and satisfies
\begin{equation*}
\lim_{t\to\infty}\upsilon_{\ij}(t)\ =\ \frac{1}{\mu}\quad\text{and}\quad\lim_{t\to-\infty}\upsilon_{\ij}(t)\ =\ 0.
\end{equation*}
\end{description}
\end{Theorem}

\begin{proof}
Fix any $i\in\cS$ and let $\pi$ once again be defined by \eqref{eq:def pi}, thus $\pi_{i}=1$. Using
\begin{align*}
\U_{i}(C)\ &=\ \int_{\R}\Erw_{i}\left(\sum_{n=0}^{\sigma(i)-1}\1_{C}(M_{n},x+S_{n})\right)\ \U_{\,ii}(dx)
\end{align*}
and Stone's decomposition for $\U_{ii}$ from the previous result, we arrive at the decompositon $\U_{i}=\U_{i}^{1}+\U_{i}^{2}$ into the finite measure
\begin{align*}
\U_{i}^{1}(C)\ :=\ \int_{\R}\Erw_{i}\left(\sum_{n=0}^{\sigma(i)-1}\1_{C}(M_{n},x+S_{n})\right)\ \U_{\,ii}^{1}(dx)
\end{align*}
with total mass $\U_{i}^{1}(\cS\times\R)=\pi_{i}^{-1}\U_{\,ii}^{1}(\R)$ and the $\sigma$-finite measure
\begin{align*}
\U_{i}^{2}(C)\ :=&\ \int_{\R}\Erw_{i}\left(\sum_{n=0}^{\sigma(i)-1}\1_{C}(M_{n},x+S_{n})\right) \upsilon_{\,ii}(x)\ \llam(dx)\\
=&\ \int_{\R}\Erw_{i}\left(\sum_{n=0}^{\sigma(i)-1}\1_{C}(M_{n},x)\upsilon_{\,ii}(x-S_{n})\right)\ \llam(dx).
\end{align*}
Choosing $C=\{j\}\times B$ for arbitrary $j\in\cS$ and $B\in\cB(\R)$, we obtain
\begin{align*}
\U_{i}^{2}(\{j\}\times B)\ =\ \int_{B}\Erw_{i}\left(\sum_{n=0}^{\sigma(i)-1}\1_{\{M_{n}=j\}}\upsilon_{\,ii}(x-S_{n})\right)\ \llam(dx)
\end{align*}
and thereby that $\U_{i}^{2}$ has $\pi\otimes\llam$-density
\begin{align*}
\upsilon_{\ij}(t)\ =\ \pi_{j}^{-1}\Erw_{i}\left(\sum_{n=0}^{\sigma(i)-1}\1_{\{M_{n}=j\}}\,\upsilon_{\,ii}(t-S_{n})\right)\quad (t\in\R),
\end{align*}
which satisfies $($with $\|\cdot\|_{\infty}$ denoting the sup-norm$)$
$$ \upsilon_{\ij}(t)\ \le\ \|\upsilon_{ii}\|_{\infty} $$
for all $j\in\cS$ and $t\in\R$, thus $\|\upsilon_{i}\|_{\infty}\le\|\upsilon_{ii}\|_{\infty}$,
and is continuous in the second argument. The remaining asymptotic assertions are now derived by using the asymptotic properties of $\upsilon_{ii}$ stated in the previous proposition and the dominated convergence theorem.\qed
\end{proof}

In the spread-out case the class of functions $g$ satisying the assertions of Theorem \ref{thm:MRT2} can be relaxed.

\begin{Theorem}[Markov renewal theorem II: spread-out case]\label{thm:MRT3}
Let $(M_{n},S_{n})_{n\ge 0}$ be spread out and $g:\cS\times\R\to\R$ a measurable function satisfying (compare \eqref{eq:dRi1} and \eqref{eq:dRi2})
\begin{align}
&g_{i}\in L^{\infty}(\llam)\text{ and }\lim_{|x|\to\infty}g_{i}(x)=0\ \text{ for all }i\in\cS,\label{eq:dRi3'}\\
&\hspace{2cm}\sum_{i\in\cS}\pi_{i}\,\|g_{i}\|_{\infty}\ <\ \infty,\label{eq:dRi4}\\
&\hspace{.6cm}g\in L^{1}(\pi\otimes\llam),\text{ i.e. }\sum_{i\in\cS}\pi_{i}\|g_{i}\|_{1}\ <\ \infty.\label{eq:dRi5'}
\end{align}
Then all assertions of Theorem \ref{thm:MRT2} about the $\U_{i}*g$ remain valid.
\end{Theorem}

\begin{proof}
Again let $g$ w.l.o.g.\ be nonnegative. Fix any $i\in\cS$, choose $\pi=\pi^{(i)}$ and use Stone's decomposition of $\U_{\,ii}$ from Prop.\ \ref{prop:Stone-type U_ij}(a) to infer
\begin{equation*}
\U_{i}*g(t)\ =\ \Erw_{i}\Bigg(\sum_{n=0}^{\sigma(i)-1}\!\!\Big(\U_{\,ii}^{1}*g_{M_{n}}(t-S_{n})+\U_{\,ii}^{2}*g_{M_{n}}(t-S_{n})\Big)\Bigg)\ =:\ J_{1}(t)+J_{2}(t)
\end{equation*}
for all $t\in\R$. Put $G(i):=\|g_{i}\|_{\infty}$ and recall that $\|\U_{\,ii}^{1}\|:=\U_{\,ii}^{1}(\R)<\infty$. It follows that $\sum_{n=0}^{\sigma(i)-1}\U_{\,ii}^{1}*g_{M_{n}}(t-S_{n})\le \|\U_{\,ii}^{1}\|\,\sum_{n=0}^{\sigma(i)-1}G(M_{n})$ $\Prob_{i}$-a.s.,
\begin{align*}
J_{1}(t)\ \le\ \|\U_{\,ii}^{1}\|\,\Erw_{i}\left(\sum_{n=0}^{\sigma(i)-1}G(M_{n})\right)\ 
=\ \|\U_{\,ii}^{1}\|\,\Erw_{\pi}G(M_{0})\ <\ \infty\quad(\text{use \eqref{eq:occupation measure formula}})
\end{align*}
(thus the boundedness of $J_{1}$) and then with the dominated convergence theorem
$$ \lim_{|t|\to\infty}J_{1}(t)\ =\ 0, $$
for $\lim_{|t|\to\infty}g_{i}(t)=0$. Left with a study of $J_{2}(t)$, we note that
\begin{align*}
J_{2}(t)\ &=\ \Erw_{i}\Bigg(\sum_{n=0}^{\sigma(i)-1}\int_{\R}g_{M_{n}}(t-x-S_{n})\upsilon_{ii}(x)\ \llam(dx)\Bigg)\\
&=\ \int_{\R} \Erw_{i}\Bigg(\sum_{n=0}^{\sigma(i)-1}g_{M_{n}}(x)\upsilon_{ii}(t-x-S_{n})\Bigg)\ \llam(dx)\\
&=\ \int_{\R}\Erw_{\pi}g_{M_{0}}(x)\,\upsilon_{ii}(t-x)\ \llam(dx)\\
&\quad+\ \int_{\R}\Erw_{i}\Bigg(\sum_{n=0}^{\sigma(i)-1}g_{M_{n}}(x)\big(\upsilon_{ii}(t-x-S_{n})-\upsilon_{ii}(t-x)\big)\Bigg)\ \llam(dx)
\end{align*}
By combining the assumptions on $g$ with the properties of $\upsilon_{ii}$, it is now straightforward to conclude that $J_{2}$ is bounded and that the first term of the last two lines converges to the asserted respective limit as $t\to\pm\infty$, while the second one converges to 0. We omit further details.\qed
\end{proof}

\section{Proofs of the main results}\label{sec:proofs}

In view of the results of the previous section and the furnishing lemmata in Section \ref{sec:setup}, it is now straightforward to deduce our main theorems.

\begin{proof}[of Theorem \ref{MRT for Q}]
As noted at the beginning of Section \ref{sec:setup}, $P=D^{-1}QD$ has essentially unique left eigenvector $\pi=u^{\top}D=(\sfu_{i}\sfv_{i})_{i\in\cS}$ associated with eigenvalue 1, so that $\pi$ is the essentially unique stationary measure of the Markov chain $(M_{n})_{n\ge 0}$ with transition matrix $P$. Moreover, the MRW $(M_{n},S_{n})_{n\ge 0}$ has stationary drift $\mu$ as defined in \eqref{eq:positive drift} under $\pi$ and is nonarithmetic if $Q\otimes F$ has this property. Hence, a combination of Lemma \ref{lem:MRM connection} and the Markov renewal theorem \ref{thm:MRT1} yields
\begin{align*}
\lim_{t\to\infty}\V_{\ij}([t,t+h])\ =\ \frac{\sfv_{i}}{\sfv_{j}}\,\lim_{t\to\infty}\U_{\ij}([t,t+h])\ =\ \frac{\sfv_{i}\pi_{j}h}{\mu \sfv_{j}}\ =\ \frac{\sfv_{i}\sfu_{j}h}{\mu}
\end{align*}
as well as $\lim_{\,t\to-\infty}\V_{\ij}([t,t+h])=0$ for all $h>0$ and $i,j\in\cS$.\qed
\end{proof}

\begin{proof}[of Theorem \ref{Stone-type & MRT for Q}]
If $Q\otimes F$ is spread out, then so is $(M_{n},S_{n})_{n\ge 0}$. Therefore, by another use of Lemma \ref{lem:MRM connection} in combination with Theorem \ref{thm:MRM spread out}, the assertions of the theorem follow directly when observing that $\V=D\,\U^{1}D^{-1}+D\,\U^{2}D^{-1}$ provides a Stone-type decomposition of $\V$. Further details can be omitted.\qed
\end{proof}

\begin{proof}[of Theorem \ref{MRT2 for Q}]
Recall that $\wh{g}(t):=D^{-1}g(t)=(\sfv_{i}^{-1}g_{i}(t))_{i\in\cS}$. Then it is easily seen that $\wh{g}$ is $\pi$-directly Riemann intgrable iff $g$ is $\sfu$-directly integrable, and that $\wh{g}$ satisfies 
\eqref{eq:dRi3'}, \eqref{eq:dRi5'} iff $g$ itself satisfies \eqref{eq:dRi3}, \eqref{eq:dRi5}.
Further observing that 
$$ \V*g\ =\ (D\,\U\,D^{-1})*D\,\wh{g}\ =\ D\,\U*\wh{g} $$
all assertions are directly inferred from Theorem \ref{thm:MRT2} or Theorem \ref{thm:MRT3}, when applied to $\U*\wh{g}$.\qed
\end{proof}

\begin{proof}[of Theorem \ref{MRE}]
The fact that $Z^{*}\in\cL$ follows directly from Theorem \ref{MRT2 for Q} so that we may immediately turn to the uniqueness assertions regarding the Markov renewal equation
\begin{equation}\label{eq:MRE}
Z\ =\ z+(Q\otimes F)*Z.
\end{equation}
Note that, if $Z$ is in $\cL$ and a solution to \eqref{eq:MRE}, then $\wh{Z}=D^{-1}Z$ is in the same class (with respect to $P$, thus replacing $\sfv$ with $(1,1,...)^{\top}$ in the definition of $\cL$) and a solution to the probabilistic counterpart of \eqref{eq:MRE}, viz.
\begin{equation}\label{eq:MRE2}
\wh{Z}\ =\ \wh{z}+(P\otimes F)*\wh{Z}.
\end{equation}
Hence we may assume w.l.o.g.\ that $Q=P$, $\sfv=(1,1,...)^{\top}$ and thus $\wh{Z}=Z$. Given any further solution $Z'\in\cL_{0}(Z^{*})$, the difference $\Delta:=Z'-Z^{*}$ is an element of $\cL_{0}$ and a solution to the homogeneous equation $\Delta=(P\otimes F)*\Delta$,
thus
$ \Delta_{i}(t)=\Erw_{i}\Delta(M_{1},t-S_{1}) $
and then upon iteration
$$ \Delta_{i}(t)\ =\ \Erw_{i}\Delta(M_{n},t-S_{n}) $$
for all $t\in\R$, $n\in\N$ and $i\in\cS$. This shows that, for all $i\in\cS$, $(\Delta(M_{n},t-S_{n}))_{n\ge 0}$ forms a bounded $\Prob_{i}$-martingale which thus converges $\Prob_{i}$-a.s.\ to a limit. But the latter equals 0, for
$$ \lim_{n\to\infty}\Delta(M_{n},t-S_{n})\ =\ \lim_{n\to\infty}\Delta(i,t-S_{\sigma_{n}(i)})\ =\ 0, $$
where as before the $\sigma_{n}(i)$ denote the a.s.\ finite return times to $i$ of the chain $(M_{n})_{n\ge 0}$. If $\cS$ is finite or $Z^{*}\in\cL_{0}$, then $\cL_{0}(Z^{*})=\cL_{0}$ and the previous argument extends to all solutions $Z'\in\cL_{0}$.\qed
\end{proof}

\begin{proof}[of Theorem \ref{Choquet-Deny}]
Given a solution $Z\in\cC_{b}$ of the homogeneous Markov renewal equation $Z=(Q\otimes F)*Z$, the function $\wh{Z}$ is a bounded, component-wise continuous solution to $\wh{Z}=(P\otimes F)*\wh{Z}$ and therefore $(\wh{Z}(M_{n},t-S_{n}))_{n\ge 0}$ a bounded $\Prob_{i}$-martingale for all $i\in\cS$. Using the Optional Sampling Theorem, it follows that
$$ \wh{Z}_{i}(t)\ =\ \Erw_{i}\wh{Z}(M_{\sigma(i)},t-S_{\sigma(i)})\ =\ \Erw_{i}\wh{Z}_{i}(t-S_{\sigma(i)}) $$
for all $i\in\cS$. In other words, $\wh{Z}_{i}$ forms a bounded, continuous solution to the ordinary Choquet-Deny equation $\wh{Z}_{i}=\wh{F}_{i}*\wh{Z}_{i}$ for each $i\in\cS$, where $\wh{F}_{i}$ denotes the law of $S_{\sigma(i)}$ under $\Prob_{i}$. Since $\wh{F}_{i}$ is nonarithmetic (Lemma \ref{lem:lattice-type}) and $\wh{Z}_{i}$ is continuous, the latter function must equal a constant $c_{i}$ (see \cite[p.\ 382]{Feller:71}). By another appeal to the Optional Sampling Theorem, now for distinct $i,j\in\cS$, we find that
$$ c_{i}\ =\ \wh{Z}_{i}(t)\ =\ \Erw_{i}\wh{Z}_{j}(t-S_{\sigma(j)})\ =\ c_{j}, $$
where $\Prob_{i}(\sigma(j)<\infty)=1$ is guaranteed by the recurrence of $(M_{n})_{n\ge 0}$. Consequently, $\wh{Z}_{i}\equiv c$ for all $i\in\cS$ and some $c\in\R$ as asserted.\qed
\end{proof}

\section{Three examples}\label{sec:examples}

Quasi-stochastic matrices arise in various areas of applied probability, typically in connection with an exponential change of measure. For illustration we present three examples here but make no attempt of complete elaboration of all technical details.

\subsection{Supremum of a Markov random walk with negative drift}\label{MMRW}

Consider a nonarithmetic MRW $(M_{n},S_{n})_{n\ge 0}$ with recurrent driving chain $(M_{n})_{n\ge 0}$ having finite state space $\cS$, transition matrix $P=(p_{\ij})_{i,j\in\cS}$ and stationary distribution $\pi$. Let $G_{\ij}$ be the conditional distribution of $X_{n}$ given $(M_{n-1},M_{n})=(i,j)$, and denote by $\phi_{\ij}(\lambda):=\int e^{\lambda x}\,G_{\ij}(dx)$ its moment generating function. We are interested in the asymptotic tail behavior of the supremum $W:=\sup_{n\ge 0}S_{n}$ under the following additional assumptions:
\begin{description}[(A2)]
\item[(B1)] The stationary drift $\mu:=\Erw_{\pi}S_{1}$ of $(M_{n},S_{n})_{n\ge 0}$ is negative.
\item[(B2)] There exists $\lambda>0$ such that the spectral radius (maximal positive eigenvalue) $\rho(Q)$ of $Q=P_{\lambda}:=(p_{\ij}\phi_{\ij}(\lambda))_{i,j\in\cS}$ is one, i.e.
$$ \lim_{n\to\infty}\left(\Erw_{i}e^{\lambda S_{n}}\right)^{1/n}\ =\ \lim_{n\to\infty}\left(\Erw_{\pi}e^{\lambda S_{n}}\right)^{1/n}\ =\ 1 $$
for all $i\in\cS$.
\end{description}
Clearly, $Q$ is quasi-stochastic with positive left and right eigenvectors $\sfu,\sfv$, respectively, satisfying \eqref{eigenvalue one} and \eqref{Perron eigenvector normalized}. Put $F_{\ij}(dx)=\phi_{\ij}(\lambda)^{-1}e^{\lambda x}\,G_{\ij}(dx)$ for $i,j\in\cS$ and define the probability measures $\Prob_{i}^{\lambda}$, $i\in\cS$, by
\begin{align*}
\Prob_{i}^{\lambda}(M_{k}=i_{k},X_{k}\le t_{k},\,1\le k\le n)\ :=\ &\frac{\sfv_{i_{n}}}{\sfv_{i}}q_{ii_{1}}\cdot...\cdot q_{i_{n-1}i_{n}}\,F_{ii_{1}}(t_{1})\cdot...\cdot F_{i_{n-1}i_{n}}(t_{n})\\
=\ &\frac{1}{\sfv_{i}}\Erw_{i}\sfv_{M_{n}}e^{\lambda S_{n}}\prod_{k=1}^{n}\1_{\{M_{k}=i_{k},X_{k}\le t_{k}\}}
\end{align*}
for $n\in\N$, $i_{1},...,i_{n}\in\cS$ and $t_{1},...,t_{n}\in\R$. Note that the last relation extends to all stopping times $\sigma$ for $(M_{n},X_{n})_{n\ge 0}$, viz.
\begin{align}
\begin{split}\label{P_i^lambda}
\Prob_{i}^{\lambda}(M_{k}=i_{k},X_{k}\le t_{k},\,&1\le k\le \sigma<\infty)\\
&=\ \frac{1}{\sfv_{i}}\Erw_{i}\sfv_{M_{\sigma}}e^{\lambda S_{\sigma}}\1_{\{\sigma<\infty\}}\prod_{k=1}^{\sigma}\1_{\{M_{k}=i_{k},X_{k}\le t_{k}\}},
\end{split}
\end{align}
and this particularly implies
$$ \Erw_{i}e^{\lambda S_{\sigma(i)}}\ =\ \Prob_{i}^{\lambda}(\sigma(i)<\infty)\ =\ 1 $$
for all $i\in\cS$. From these settings, it follows easily that $(M_{n},S_{n})_{n\ge 0}$ is still a MRW under the $\Prob_{i}^{\lambda}$. Its recurrent driving chain has transition matrix $D^{-1}QD$ with $D=\text{diag}(\sfv_{i},i\in\cS)$ as in the previous sections and stationary distribution $\wh{\pi}=(\sfu_{i}\sfv_{i})_{i\in\cS}$. The conditional law of $X_{n}$ given $(M_{n-1},M_{n})=(i,j)$ is now $F_{\ij}$. Assumption (B1) in combination with the convexity of $\alpha\mapsto\log\rho(P_{\alpha})$ further implies that the stationary drift $\Erw_{\wh{\pi}}^{\lambda}S_{1}$ is positive.

\vspace{.1cm}
Turning to $W=\sup_{n\ge 0}S_{n}$, let us define its tail function $H_{i}(t):=\Prob_{i}(W>t)\1_{[0,\infty)}(t)$ for $i\in\cS$. In the special case of i.i.d. $X_{1},X_{2},...$, it is a well-known fact that $W$ forms a solution to Lindley's equation
$$ W\ \stackrel{d}{=}\ (X+W)^{+}, $$
where $X$ forms a copy of the $X_{n}$ and is independent of $W$. In the presence of a Markovian environment as assumed here, the equation takes the more general form
$$ W\ \stackrel{d}{=}\ (X_{1}+W_{1})^{+}\quad\text{under }\Prob_{i} $$
for each $i\in\cS$, where $\Prob_{i}(W_{1}\in\cdot|M_{1}=j,X_{1}=x)=\Prob_{j}(W\in\cdot)$ for all $j\in\cS$, and from this equation we infer the Wiener-Hopf-type integral equations
\begin{equation}\label{WHIE for W}
H_{i}(t)\ =\ \Prob_{i}(X_{1}>t)\ +\ \sum_{j\in\cS}p_{\ij}\int_{(-\infty,t]}H_{j}(t-x)\ G_{\ij}(dx).
\end{equation}
for $i\in\cS,\,t\ge 0$. On the other hand, one can also easily verify that
\begin{align}
\begin{split}\label{MRE for W}
H_{i}(t)\ &=\ \Prob_{i}(S_{\sg}>t,\sg<\infty)\ +\ \Erw_{i}H_{M_{\sg}}(t-S_{\sg})\1_{\{\sg<\infty\}}\\
&=\ \Prob_{i}(S_{\sg}>t,\sg<\infty)\ +\ \sum_{j\in\cS}p_{\ij}^{>}\int_{(0,t]}H_{j}(t-x)\ G_{\ij}^{>}(dx)
\end{split}
\end{align}
for $i\in\cS,\,t\ge 0$, where $\sg:=\inf\{n\ge 1:S_{n}>0\},\,p_{\ij}^{>}:=\Prob_{i}(M_{\sg}=j,\sg<\infty)$, and
$$ G_{\ij}^{>}\ :=\ \Prob_{i}(S_{\sg}\in\cdot|M_{\sg}=j,\sg<\infty). $$
\eqref{MRE for W} is a one-sided Markov renewal equation of defective type, for (B1) entails $\Prob_{i}(\sg<\infty)<1$ for at least one $i\in\cS$. Multiplying it with $e^{\lambda t}$ and defining 
\begin{align*}
&q_{\ij}^{>}:=p_{\ij}^{>}\int e^{\lambda x}G_{\ij}^{>}(dx),\quad
F_{\ij}^{>}(dx)=(\int e^{\lambda y}G_{\ij}^{>}(dy))^{-1}e^{\lambda x}\,G_{\ij}^{>}(dx),\\
&Z_{i}(t):=e^{\lambda t}H_{i}(t)\quad\text{and}\quad
z_{i}(t):=e^{\lambda t}\Prob_{i}(S_{\sg}>t,\sg<\infty)\1_{[0,\infty)}(t)
\end{align*}
we obtain
\begin{equation}\label{transformed MRE for W}
Z_{i}(t)\ =\ z_{i}(t)\ +\ \sum_{j\in\cS}q_{\ij}^{>}\int_{(0,t]}Z_{j}(t-x)e^{\lambda x}\ F_{\ij}^{>}(dx)
\end{equation}
for $i\in\cS,\,t\ge 0$, thus $Z=z+(Q^{>}\otimes F^{>})*Z$ with $Q^{>},F^{>}$ having obvious meanings. By \eqref{P_i^lambda},
\begin{align}\label{Q> quasi-stochastic}
\sum_{j\in\cS}q_{\ij}^{>}\sfv_{j}\ =\ \Erw_{i}\sfv_{M_{\sg}}e^{\lambda S_{\sg}}\ =\ \sfv_{i}\,\Prob_{i}^{\lambda}(\sg<\infty)\ =\ \sfv_{i}
\end{align}
for all $i\in\cS$, which shows that $Q^{>}$ has maximal eigenvalue 1 and is therefore quasi-stochastic if it is also irreducible. The latter need not be true, but it can be shown that $Q^{>}$ is irreducible on $\cS':=\{i:\sfu_{i}^{>}>0\}$, where $\sfu^{>}$ denotes the unique left eigenvector of $Q^{>}$ satisfying $\sum_{i\in\cS}\sfu_{i}^{>}=1$. We omit a further discussion of this issue and just mention that $\cS'$ is in fact the maximal irreducibility class of the strictly ascending ladder chain $(M_{\sigma_{n}^{>}})_{n\ge 0}$ associated with $(M_{n},S_{n})_{n\ge 0}$ where $\sg_{1}=\sg$ and $\sg_{n}:=\inf\{k>\sg_{n-1}:S_{\sg_{n-1}+k}>S_{\sg_{n-1}}\}$ for $n\ge 2$. Let us also mention here that
$(M_{\sg_{n}},S_{\sg_{n}})_{n\ge 0}$ is again nonarithmetic (see \cite{Alsmeyer:00}) which in turn implies that $F^{>}$ is nonarithmetic.
Since $\Erw_{i}e^{\lambda S_{\sg}}<\infty$ for each $i\in\cS$ and thus also
$$ \Erw_{\sfu^{>}}e^{\lambda S_{\sg}}\ =\ \sum_{i\in\cS'}\sfu_{i}^{>}\int_{(0,\infty)}\lambda e^{\lambda t}\,\Prob_{i}(S_{\sg}>t)\ dt\ <\ \infty $$
follows from \eqref{Q> quasi-stochastic} and $\sfv_{*}:=\min_{i\in\cS}\sfv_{i}>0$,
one can easily verify (see \cite[Lemma 3.6.2]{Alsmeyer:91} for a similar argument) that
$z$ is $\sfu^{>}$-directly Riemann integrable, in particular bounded. With the help of our results in Section \ref{sec:intro} we may finally conclude that $Z_{i}(t)=e^{\lambda t}\,\Prob_{i}(W>t)$ exists and is finite for each $i\in\cS'$. The same may indeed be shown for all $i\in\cS$. A further derivation of the form of these limits is omitted, but we refer to \cite[Ch. VI]{Asmussen:00} for a similar and more extensive treatment in the context of collective risk theory. The tail behavior of $W$ for a different regime is studied in \cite{AlsSgib:99} by combining Banach algebra and Wiener-Hopf factorization techniques.

In principle the previous considerations remain valid if the modulating chain has infinite state space $\cS$. However, the quasi-stochasticity of $Q$ and $Q^{>}$ is a more difficult matter and thus requires additional arguments because we cannot resort to Perron-Frobenius theory for (finite) nonnegative matrices.

\subsection{Age-dependent multitype branching processes}\label{subsec:multitype}

This is an example from the class of multi-type Crump-Mode-Jagers processes. We refer to the Mode's book \cite[Chapter 3]{Mode:71a} for more detailed information and further mention an article by the same author about a related model used for cell-cycle analysis \cite{Mode:71b}.

\vspace{.1cm}
Consider a population stemming from one ancestor born at time 0 which may be of any type $s\in\cS=\{1,...,m\}$. At the end of its life, each individual of type $i$ gives birth to a random number of offspring of type $j$ with finite mean $\mu_{\ij}$ for any $j\in\cS$ and has a nonarithmetic lifetime distribution $G_{i}$ on $(0,\infty)$. Moreover, all individuals behave independently. We are interested in the asymptotic behavior of $S(t)=(S_{\ij}(t))_{i,j\in\cS}$,
where $S_{\ij}(t)$ denotes the mean number of type $j$ individuals alive at time $t\ge 0$ when starting from one individual of type $i$. For simplicity, let the numbers of offspring be independent of the lifetime of an individual. Put $\overline{G}_{i}:=1-G_{i}$. Then a standard renewal argument leads to
\begin{equation*}
S_{\ij}(t)\ =\ \delta_{\ij}\overline{F}_{i}(t)+\sum_{k=1}^{m}\mu_{ik}\int_{(0,t]}S_{kj}(t-x)\ G_{k}(dx)\quad (t\ge 0)
\end{equation*}
for all $1\le i,j\le m$, that is $S=g+(M\otimes G)*S$ with $M:=(\mu_{\ij})_{1\le i,j\le m}$,
$$ g(t)\ :=\ \begin{pmatrix}
\overline{G}_{1}(t)&&0\\ &\ddots&\\ 0&&\overline{G}_{m}(t)
\end{pmatrix}
\quad\text{and}\quad 
G(t)\ :=\ \begin{pmatrix}
G_{1}(t)&\ldots&G_{1}(t)\\ &\ddots&\\ G_{m}(t)&\ldots&G_{m}(t)
\end{pmatrix} $$
Here $S(t),z(t)$ are matrices instead of vectors, but we may of course consider their column vectors $S_{\bullet j}(t)=(S_{\ij}(t))_{1\le i\le m},\,g_{\bullet j}(t)=(\delta_{\ij}\overline{G}_{j}(t))_{1\le i\le m}$ separately, or any linear combination $\sfv^{\top}S(t)=\sum_{j=1}^{m}\sfv_{j}S_{\bullet j}(t)$.

Now consider $\alpha\in\R$ such that $\phi_{i}(\alpha):=\int e^{-\alpha t}\,G_{i}(dt)<\infty$ for each $i=1,...,m$. Defining $Z(t):=e^{-\alpha t}S(t)$, we then find that $Z=z+(Q\otimes F)*Z$ with $z(t):=e^{-\alpha t}g(t)$, $Q:=(m_{\ij}\phi_{i}(\alpha))_{1\le i,j\le m}$, and
$$ F(t)\ :=\ \begin{pmatrix}
F_{1}(t)&\ldots&F_{1}(t)\\ &\ddots&\\ F_{m}(t)&\ldots&F_{m}(t)
\end{pmatrix},\quad\text{where }F_{i}(t)\ :=\ \phi_{i}(\alpha)^{-1}\int_{[0,t]}e^{-\alpha x}\ G_{i}(dx). $$
If $\alpha$ -- called Malthusian parameter of the population -- can be chosen such that $Q$ has maximal eigenvalue $1$ and is primitive, thus $Q^{n}$ a strictly positive matrix for some $n\in\N$ (see \cite{Seneta:06}), then the results of Section \ref{sec:intro} can be utilized to determine the limit of $e^{-\alpha t}S(t)$ as $t\to\infty$. In principle, these considerations may be extended to the case of infinite type space $(\cS=\N)$ in the sense that the above Markov renewal equations remain valid. On the other hand, as in the previous example, the quasi-stochasticity of $Q$ including the therefore necessary existence of the Malthusian parameter $\alpha$ is more delicate.

\vspace{.1cm}
We finally note that other functionals of the described population may be studied in a similar manner. For example, if $A_{\ij}(t)$ denotes the average total age of all type $j$ individuals alive at time $t$ when the ancestor of the population is of type $i$, then it is readily verified that $A(t)=(A_{\ij}(t))_{1\le i,j\le m}$ satisfies the Markov renewal equation $A=f+(M\otimes G)*A$ with $M,G$ as before and
$$ f(t)\ :=\ \left(\delta_{\ij}\,t\overline{G}_{i}(t)\right)_{1\le i,j\le m}. $$

\subsection{Random difference equations in Markovian environment}\label{subsec:RDE}

Let $(A_{n},B_{n})_{n\in\Z}$ be a doubly infinite stationary ergodic sequence and consider  the random difference equation
\begin{equation}\label{RDE}
Y_{n}\ =\ A_{n}Y_{n-1}+B_{n}
\end{equation}
for $n\ge 0$. It was shown by Brandt \cite{Brandt:86} that, if
\begin{equation}\label{eq:RDE moment}
\Erw\log|A_{0}|<0\quad\text{and}\quad\Erw\log^{+}|B_{0}|<\infty,
\end{equation}
then a stationary solution of $(Y_{n})_{n\ge 0}$ exists and may be realized by defining
$$ Y_{0}\ =\ B_{0}+\sum_{n\ge 0}A_{-n}\cdot...\cdot A_{0}\,B_{-n-1}. $$
Regarding the existence and properties of the stationary law of $Y_{0}$ (often called \emph{perpetuity}), many papers have dealt with the situation when the $(A_{n},B_{n})$ are i.i.d.\ and possibly multivariate, see \cite{Vervaat:79,Goldie:91,GolMal:00,AlsIksRoe:09,Kesten:73,LePage:83,BurDamGui:09,AlsMen:12}. The case when $(A_{n})_{n\in\Z}$ forms an irreducible stationary Markov chain taking values in a finite subset $\cS$ of $\R$ and the $B_{n}$ are i.i.d.\ and independent of the $A_{n}$ was treated by de Saporta \cite{Saporta:05b}, see also \cite{Roitershtein:07,RoiterZhong:13} for the more general case of continuous state space $\cS$.

\vspace{.1cm}
Let us take a closer look at the situation treated in \cite{Saporta:05b}, for simplicity confining ourselves to the case when $\cS\subset (0,\infty)$, but allowing that $\cS$ is an infinite countable set. Denote by $P=(p_{ss'})_{s,s'\in\cS}$ the transition matrix of $(A_{n})_{n\ge 0}$ and by $\pi=(\pi_{s})_{s\in\cS}$ its unique stationary distribution.
Note that the dual backward chain $(A_{-n})_{n\ge 0}$ has transition probabilities $\wh{p}_{ss'}=\pi_{s'}p_{s's}/\pi_{s}$.

Being interested in $\Prob(\pm Y_{1}>t,A_{0}=s)$ for $(s,t)\in\cS\times\R$, observe that, by \eqref{RDE},
\begin{equation*}
\Prob(\pm Y_{1}>t,A_{1}=s)\ =\ \Prob(\pm sY_{0}>t,A_{1}=s)+\psi_{s}^{\pm}(t),
\end{equation*}
where
\begin{equation*}
\psi_{s}^{\pm}(t)\ :=\ \Prob(\pm sY_{0}+B_{1}>t,A_{1}=s)-\Prob(\pm sY_{0}>t,A_{1}=s).
\end{equation*}
For $\alpha$ still to be specified, define the smoothed tail functions
$$ Z_{s}^{\pm}(t)\ :=\ \frac{1}{\pi_{s}e^{t}}\int_{0}^{e^{t}}u^{\alpha}\,\Prob(\pm sY_{1}>u,A_{1}=s)\ du $$
and $z_{s}^{\pm}(t):=\pi_{s}^{-1}e^{-t}\int_{0}^{e^{t}}u^{\alpha}\,\psi_{s}^{\pm}(u)\,du$. Put $F_{ss'}(t):=\1_{[\log s,\infty)}(t)$. Then it is not difficult to show (see \cite[Section 3]{Saporta:05b} for details) that
\begin{equation*}
Z_{s}^{\pm}(t)\ =\ z_{s}^{\pm}(t)\ +\ s^{\alpha}\sum_{s'\in\cS}\wh{p}_{ss'}F_{ss'}*Z_{s'}^{\pm}(t)
\end{equation*}
for all $(s,t)\in\cS\times\R$. It follows that $Z^{+}(t)=(Z_{s}^{+}(t))_{s\in\cS}$ and $Z^{-}(t)=(Z_{s}^{-}(t))_{s\in\cS}$ both satisfy the Markov renewal equation $Z=z+(Q\otimes F)*Z$ with $z(t)=z^{+}(t)=(z_{s}^{+}(t))_{s\in\cS}$ and $z(t)=z^{-}(t)=(z_{s}^{-}(t))_{s\in\cS}$, respectively, and with
$$ Q\ =\ \left(s^{\alpha}\wh{p}_{ss'}\right)_{s,s'\in\cS}\ =\ \left(s^{\alpha}\pi_{s'}p_{s's}/\pi_{s}\right)_{s,s'\in\cS}. $$
Therefore the asymptotic behaviour of $Z^{+}(t)$ and $Z^{-}(t)$ as $t\to\infty$ can be determined with the help of the results in Section \ref{sec:intro} if (besides further technical assumptions) we can choose $\alpha>0$ such that $Q$ is quasi-stochastic which particularly requires that $Q$ has spectral radius one, i.e.
$$ \rho(Q)\ =\ \lim_{n\to\infty}\left(\Erw(A_{0}\cdot...\cdot A_{n-1})^{\alpha}\right)^{1/n}\ =\ 1. $$
In the case of finite $\cS$, the latter already implies quasi-stochasticity as a consequence of the Perron-Frobenius theorem, but for infinite state space this needs further inspection.

\bibliographystyle{abbrv}
\bibliography{StoPro}

\end{document}